\documentclass[12pt,reqno]{article}

\usepackage{authblk}

\usepackage{amsmath, mathrsfs, amssymb, amsthm}
\usepackage{enumerate}

\newtheorem{defi}{Definition}[section]
\newtheorem{remark}[defi]{Remark}
\newtheorem{thm}[defi]{Theorem}
\newtheorem{lemma}[defi]{Lemma}
\newtheorem{proposition}[defi]{Proposition}
\newtheorem{corollary}[defi]{Corollary}

\providecommand{\keywords}[1]{\textbf{Keywords} - #1}
\providecommand{\subjclass}[1]{\textbf{2000 Mathematics Subject Classification 2000} - #1}

\begin{document}

\title{Amenability and Orlicz Figa-Talamanca Herz algebras}

\author[1]{Rattan Lal}
\author[1]{N. Shravan Kumar\thanks{Corresponding Author: shravankumar@maths.iitd.ac.in}}

\affil[1]{Department of Mathematics, Indian Institute of Technology Delhi, Delhi - 110016, India.}

\date{}

\maketitle

\begin{abstract}
In this paper, we characterize the amenablity of locally compact groups in terms of the properties of the Orlicz Figa-Talamanca Herz algebras.
\end{abstract}

\keywords{Orlicz Space, Orlicz Figa-Talamanca Herz algebra, Amenable group, Bounded approximate identity, Derivation, Splittings}

\subjclass{Primary 43A07, 43A15; Secondary 46J10}

\section{Introduction}
Let $G$ be a locally compact group and let $A_p(G)$ the Figa-Talamanca Herz algebra introduced by Herz \cite{H1}. The following theorem on the characterization of amenability in terms of the $A_p(G)$ algebras is well-known.
\begin{thm}
Let $G$ be a locally compact group. Then the following are equivalent:
\begin{enumerate}[a)]
\item The group $G$ is amenable.
\item The Banach algebra $A_p(G)$ possesses a bounded approximate identity.
\item Every closed cofinite ideal is of the form $I(E),$ where $E$ is a finite subset of $G.$
\item The Banach algebra $A_p(G)$ factorizes weakly.
\item Each homomorphism from $A_p(G)$ with finite dimensional range is continuous.
\item Every derivation of $A_p(G)$ into a Banach $A_p(G)$-bimodule is continuous.
\end{enumerate}
\end{thm}
The equivalence of the statements a) and b) was due to Herz \cite{H2}. The equivalence of the statements a), c), e) and f) were due to Forrest \cite{F3}. The equivalence of the statements a) and d) was due to Losert \cite{L1}.

In \cite{LK}, we have introduced and studied the $L^\Phi$-versions of the Figa-Talamanca Herz algebras. Here $L^\Phi$ denotes the Orlicz space corresponding to the Young function $\Phi.$ The space $A_\Phi(G)$ is defined as the space of all continuous functions $u,$ where $u$ is of the form $$u=\underset{n=1}{\overset{\infty}{\sum}}f_n*\check{g_n},$$ where $f_n\in L^\Phi(G),$ $g_n\in L^\Psi(G),$ $(\Phi,\Psi)$ is a pair of complementary Young functions satisfying the $\Delta_2$-condition and $$\underset{n=1}{\overset{\infty}{\sum}}N_\Phi(f_n)\|g_n\|_\psi<\infty.$$ 

This paper has the modest aim of proving the above said equivalent statements in the context of $A_\Phi(G)$ algebras. We shall begin with some preliminaries that are needed in the sequel.

\section{Preliminaries}
Let $ \Phi: \mathbb{R}\rightarrow [0,\infty] $ be a convex function. Then $\Phi$ is called a Young function if it is symmetric and  satisfies $ \Phi(0)= 0 $ and $\underset{x\rightarrow \infty}{\lim}  \Phi(x)= + \infty $. If $ \Phi$ is any Young function, then define $\Psi$ as
$$ \Psi(y):= \sup{\{x|y|-\Phi(x) : x\geq0\}} , y\in\mathbb{R}.$$ Then $\Psi$ is also a Young function and is termed as the complementary function to $ \Phi.$ Further, the pair $ (\Phi,\Psi) $ is called a complementary pair of Young functions. 

Let $G$ be a locally compact group with a left Haar measure $dx.$ We say that a  Young function $ \Phi $  satisfies  the $\Delta_{2}$-condition, denoted $\Phi\in\Delta_{2},$ if there exists a constant $ K>0 $ and $ x_{0} > 0$ such that $\Phi(2x)\leq K\Phi(x)$ whenever $x\geq x_{0}$ if $G$ is compact and the same inequality holds with $ x_{0}=0 $ if $G$ is non compact.

The Orlicz space, denoted $ L^{\Phi}(G),$ is a vector space consisting of measurable functions, defined as $$ {L}^{\Phi}(G) = \left\{ f: G \rightarrow  \mathbb{C}: \mbox{f is measurable and }\int_G\Phi(\beta |f|)\ dx <\infty \text{ for some}~ \beta>0  \right\} $$ The Orlicz space $ L^{\Phi}(G) $ is a Banach space when equipped with the norm
$$N_{\Phi}(f) = \inf \left\{ k>0 :\int_G\Phi\left(\frac{|f|}{k}\right) dx \leq1 \right\}.$$ The above norm is called as the Luxemburg norm or Gauge norm.  If $(\Phi,\Psi)$ is a complementary Young pair, then there is a norm on $L^\Phi(G),$ equivalent to the Luxemberg norm, given by, $$ \|f\|_{\Phi} =\sup \Bigg \{ \int_{G}|fg|dx : \int_{G}\Psi(|g|)dx\leq1 \Bigg\}.$$ This norm is called as the Orlicz norm. 

Let $C_{c}(G)$ denote the space of all continuous functions on $ G $ with compact support. If a Young function $ \Phi $  satisfies  the $\Delta_{2}$ -condition, then $C_c(G)$ is dense in $L^\Phi(G).$ Further, if the complementary function $\Psi$ is such that $\Psi$ is continuous and $\Psi(x)=0$ iff $x=0,$ then the dual of $ (L^{\Phi}(G),N_{\Phi}(\cdot)) $ is isometrically isomorphic to $ (L^{\Psi}(G),\|\cdot\|_{\Psi}).$ In particular, if both $\Phi$ and $\Psi$ satisfies the $\Delta_2$-condition, then $L^\Phi(G)$ is reflexive.

We say that an Young function $\Phi$ satisfies the Milnes-Akimovi$\check{\mbox{c}}$ condition (in short MA-condition) if for each $\epsilon>0$ there exists $c_\epsilon>1$ and an $x_1(\epsilon)\geq 0$ such that $$\Phi^\prime((1+\epsilon)x)\geq c_\epsilon\Phi^\prime(x),\ x\geq x_1(\epsilon).$$ This condition will be used again and again in many of the theorems because of the following result due to M. M. Rao \cite[Theorem 8]{Rao2}.
\begin{thm}\label{AEqC}
Let $G$ be a locally compact group. Then $G$ is amenable if and only if for each $N$-function $\Phi$ satisfying the MA-condition and for each $\nu\in M_1^+(G),$ the operator $T_\nu:L^\Phi\rightarrow L^\Phi$ has norm 1, where $T_\nu(f)=\nu*f.$
\end{thm}

\noindent For more details on Orlicz spaces, we refer the readers to \cite{RR}.

Let $\Phi$ and $\Psi$ be a pair of complementary Young functions satisfying the $\Delta_2$ condition. Let $$A_\Phi(G)=\left\{u=\underset{n=1}{\overset{\infty}{\sum}}f_n*\check{g_n}:\{f_n\}\subset L^\Phi(G),\{g_n\}\in L^\Psi(G)\mbox{ and }\underset{n=1}{\overset{\infty}{\sum}}N_\Phi(f_n)\|g_n\|_\Psi<\infty\right\}.$$ Note that if $u\in A_\Phi(G)$ then $u\in C_0(G).$ If $u\in A_\Phi(G),$ define $\|u\|_{A_\Phi}$ as $$\|u\|_{A_\Phi}:=\inf\left\{\underset{n=1}{\overset{\infty}{\sum}}N_\Phi(f_n)\|g_n\|_\Psi:u=\underset{n=1}{\overset{\infty}{\sum}}f_n*\check{g_n}\right\}.$$ The space $A_\Phi(G)$ equipped with the above norm and with the pointwise addition and multiplication becomes a commutative Banach algebra \cite[Theorem 3.4]{LK}. In fact, $A_\Phi(G)$ is a commutative, regular and semisimple banach algebra with spectrum homeomorphic to $G$ \cite[Corollary 3.8]{LK}. This Banach algebra $A_\Phi(G)$ is called as the Orlicz Fig\`{a}-Talamanca Herz algebra.

Let $ \mathcal{B}(L^{\Phi}(G))$ be the linear space of all bounded linear operators on $L^{\Phi}(G)$ equipped with the operator norm. For a bounded complex Radon measure $\mu$  on $G$ and $f\in L^{\Phi}(G),$ define $T_\mu:L^\Phi(G)\rightarrow L^\Phi(G)$ by $ T_{\mu}(f)=\mu*f.$ It is clear that $T_\mu\in \mathcal{B}(L^{\Phi}(G)).$ Let $ PM_{\Phi}(G)$ denote the closure of $\{T_{\mu}:\mu\mbox{ is a bounded complex Radon measure}\}$ in $\mathcal{B}(L^{\Phi}(G))$ with respect to the ultraweak topology. It is proved in \cite[Theorem 3.5]{LK}, that for a locally compact group $G,$ the dual of $ A_{\Phi}(G)$ is isometrically isomorphic to $PM_{\Psi}(G).$ 

Let $\mathcal{A}$ be a regular, semisimple, commutative Banach algebra with the Gelfand structure space $\Delta(\mathcal{A}).$ For a closed ideal $I$ of $\mathcal{A},$ the zero set of $I,$ denoted by $Z(I),$ is a closed subset of $\Delta(\mathcal{A})$ defined as $$Z(I)=\{x\in E:\widehat{a}(x)=0\ \forall\ a\in I\}.$$ For a closed subset $E\subset\Delta(\mathcal{A}),$ we define the following ideals in $\mathcal{A}:$
\begin{eqnarray*}
j_{\mathcal{A}}(E) &=& \{a\in\mathcal{A}:\widehat{a}\mbox{ has compact support disjoint from E} \}\\
J_{\mathcal{A}}(E) &=& \overline{j_{\mathcal{A}}(E)} \\
I_{\mathcal{A}}(E) &=& \{a\in\mathcal{A}: \widehat{a} = 0\mbox{ on }E\}.
\end{eqnarray*}
Note that $J_{\mathcal{A}}(E)$ and $I_{\mathcal{A}}(E)$ are closed ideals in $\mathcal{A}$ with the zero set equal to $E$ and $j_{\mathcal{A}}(E)\subseteq I \subseteq I_{\mathcal{A}}(E)$ for any ideal $I$ with zero set $E.$ $E$ is said to be a {\it set of spectral synthesis} (or a \textit{spectral set}) for $\mathcal{A}$ if $I_{\mathcal{A}}(E) = J_{\mathcal{A}}(E).$ Let $I^c_\mathcal{A}(E)$ denote the elements in $I_\mathcal{A}(E)$ with compactly supported Gelfand transforms. We say that $E$ is a set of local spectral synthesis if $I^c_\mathcal{A}(E)\subseteq J_\mathcal{A}(E).$ By \cite[Theorem 3.6]{LK} singletons are sets of spectral synthesis for $A_\Phi(G).$ Further, every closed subgroup is a set of local synthesis for $A_\Phi(G).$ 

The closed set $E$ is a \textit{Ditkin set} if for every $u\in I_{\mathcal{A}}(E),$ there exists a sequence $\{u_n\}\subset j_{\mathcal{A}}(E)$ such that $u.u_n$ converges in norm to $u;$ if the condition holds for every compactly supported $u\in I_{\mathcal{A}}(E)$ then $E$ is called a \textit{local Ditkin set}. If the sequence can be chosen in such a way that it is bounded and is the same for all $u\in I_{\mathcal{A}}(E),$ then we say that $E$ is a \textit{strong Ditkin set.} Note that every Ditkin set is a set of spectral synthesis. The Banach algebra $\mathcal{A}$ is called a strong Ditkin algebra if all the singletons and the empty set are strong Ditkin sets.

For more on spectral synthesis see \cite{Kan, Rei}.

Throughout this paper, $G$ will denote a locally compact group and $(\Phi,\Psi)$ will denote a complementary pair of Young functions satisfying the $\Delta_2$-condition. 

\section{Amenability and bounded approximate identities}
We begin this section with the main result of this paper on the characterization of amenable groups in terms of the existence of bounded approximate identities in $A_\Phi(G).$
\begin{thm}\label{ABAI}
Let $G$ be a locally compact group and let $\Phi$ satisfy the MA condition. Then $G$ is amenable if and only if $A_\Phi(G)$ posseses a bounded approximate identity.
\end{thm}
\begin{proof}
Suppose that $G$ is amenable. Let $K$ be a compact subset of $G$ and let $\epsilon>0.$ It follows from Leptin's condition \cite[Definition 7.1]{P} that there exists a compact set $C$ in $G$ of non-zero measure such that $|KC|<(1+\epsilon)|C|.$ Let $u_{K,\epsilon} =\frac{1}{(1+\epsilon)|C|}\chi_{KC}*\check \chi_{C}.$ Then $u_{K,\epsilon}\in A_{\Phi}(G)$ and
\begin{align*}
\|u_{K,\epsilon}\|_{A_{\Phi}} 
\leq & \frac{1}{(1+\epsilon)|C|}N_{\Phi}(\chi_{KC})\|\chi_{C}\|_\Psi
\\\leq & \frac{2}{(1+\epsilon)|C|}N_{\Phi}(\chi_{KC})N_{\Psi}(\chi_{C})
\\ \leq & \frac{2}{(1+\epsilon)|C|}\left[\Phi^{-1}\left(\frac{1}{|KC|}\right)\right]^{-1}\left[\Phi^{-1}\left(\frac{1}{|C|}\right)\right]^{-1}
\\\leq & \frac{2}{(1+\epsilon)|C|}\left[\Phi^{-1}\left(\frac{1}{(1+\epsilon)|C|}\right)\right]^{-1}\left[\Phi^{-1}\left(\frac{1}{(1+\epsilon)|C|}\right)\right]^{-1}
\\ < & ~ 2 
\end{align*}
Consider the set $\Lambda=\{(K,\epsilon):K\mbox{ is a compact subset of }G\mbox{ and }\epsilon>0\}$ directed as follows: $(K_1,\epsilon_1)\prec(K_2,\epsilon_2)$ if $K_1\subset K_2$ and $\epsilon_2<\epsilon_1.$ Now consider the net $\{u_{K,\epsilon}\}_{(K,\epsilon)\in\Lambda}$ in $A_\Phi(G).$ We now claim that $\{u_{K,\epsilon}\}_{(K,\epsilon)\in\Lambda}$ is an approximate identity for $A_\Phi(G).$ Let $f \in A_\Phi(G)\cap C_c(G)$ be such that $suppf=K$ and let $\epsilon >0.$ Then $(u_{K,\epsilon}f)(x)= \frac{f(x)}{1+\epsilon}$ if $x\in K$ and $0$ otherwise.	Therefore $$\|u_{K,\epsilon}f-f\|_{A_\Phi}=\frac{\epsilon}{1+\epsilon}\|f\|_{A_\Phi} \leq \epsilon \|f\|_{A_\Phi}.$$ 

We now proceed further to prove the converse. Suppose that $A_\Phi(G)$ posseses an approximate identity $\{u_\alpha\}_{\alpha\in\Lambda}$ bounded by $c,$ for some $c>0.$ For a positive function $\psi\in C_c(G),$ using \cite[Theorem 3.5]{LK}, it can be shown as in \cite[Theorem 10.4]{P}, that $\|\psi\|_1=\|L_\psi\|_{CV_\Psi(G)}.$ We now show that a similar equality holds if we replace $\psi$ by a positive measure having compact support. Let $\mu\in M(G)$ be a positive measure having compact support. Choose $f_0\in C_c(G)$ such that $f_0$ is positive, $\check{f_0}=f_0$ and $\|f_0\|_1=1.$ Note that, for every $f\in C_c(G)$ with $N_\Phi(f)\leq 1,$ we have $f_0*f$ is also positive and has compact support. Further, $$N_\Phi(f_0*f)\leq \|f_0\|_1N_\Phi(f)\leq 1.$$ Also, if $f^\prime\in C_c^+(G),$ then $$\langle f^\prime*(f_0*f\check{)},\mu\rangle=\langle f^\prime*\check{f},\mu*f_0\rangle.$$ Thus 
\begin{eqnarray*}
\|L_\mu\|_{CV_\Psi(G)} &\geq& \sup\{\langle f^\prime*(f_0*f\check{)},\mu\rangle:f,f^\prime\in C_c^+(G)\mbox{ with }N_\Phi(f)\leq 1,N_\Psi(f^\prime)\leq 1\} \\ &=& \sup\{\langle f^\prime*\check{f},\mu*f_0\rangle:f,f^\prime\in C_c^+(G)\mbox{ with }N_\Phi(f)\leq 1,N_\Psi(f^\prime)\leq 1\} \\ &=& \|L_{\mu*f_0}\|_{CV_\Psi(G)} =\|\mu*f_0\|_1=\|\mu\|_{M(G)}.
\end{eqnarray*}
As $\|L_\mu\|_{CV_\Psi(G)}\leq\|\mu\|_{M(G)}$ for all $\mu\in M(G),$ we have $\|L_\mu\|_{CV_\Psi(G)}=\|\mu\|_{M(G)}$ for all positive $\mu$ having compact support. Thus $G$ is amenable, thanks to Theorem \ref{AEqC}.
\end{proof}
\begin{remark}
Note that the condition that $\Phi$ satisfies the MA condition in the above theorem is needed only while proving the converse, i.e., while invoking Theorem \ref{AEqC}. As mentioned in \cite{Rao2}, the assumption that $\Phi$ satisfies the MA condition is needed only to avoid the Riesz-convexity theorem. Note that the proof of the above theorem for the $A_p(G)$-algebras uses the Riesz-convexity theorem. Although an extended Riesz-convexity theorem for Orlicz spaces is available, it cannot be used here.
\end{remark}
We now begin to prove some corollaries. In the first corollary, we characterize amenability in terms of certain weak*-closed $A_\Phi(G)$-submodules of $PM_\Psi(G).$
\begin{corollary}
Let $G$ be a locally compact group, $\Phi$ satisfy the MA-condition and let $X$ be a weak*-closed $A_\Phi(G)$-submodule of $PM_\Psi(G).$ Then $G$ is amenable if and only if the following statements about $X$ are equivalent:
\begin{enumerate}[a)]
\item The space $X$ is invariantly complemented
\item The space $^\perp X$ has a bounded approximate identity.
\end{enumerate} 
\end{corollary}
\begin{proof}
The proof of the if part follows from Theorem \ref{ABAI} and \cite[Proposition 6.4]{F1}. The only if part follows again from Theorem \ref{ABAI} by choosing $X=\{0\}.$
\end{proof}
Let $B_{\Phi}(G) =\left \{ u\in C(G) : uv\in A_{\Phi}(G)~\forall~ v\in A_{\Phi}(G) \right \}.$ Then the space $B_\Phi(G)$ when equipped with the operator norm becomes a commutative banach algebra.
\begin{corollary}
Let $G$ be an amenable group and let $\Phi$ satisfy the MA-condition. Then the two norms $\|\cdot\|_{A_\Phi(G)}$ and $\|\cdot\|_{B_\Phi(G)}$ are equivalent.
\end{corollary}
\begin{proof}
By definition of $B_\Phi(G),$ it is clear that, for any $u\in A_\Phi(G),$ $\|u\|_{B_\Phi(G)}\leq\|u\|_{A_\Phi(G)}.$ For this inequality, the assumption on the group to be amenable is not needed.

For the other inequality, note that, since $G$ is amenable, by Theorem \ref{ABAI}, $A_\Phi(G)$ possesses a bounded approximate identity $\{u_\alpha\}$ such that $\|u_\alpha\|_{A_\Phi(G)}\leq 2\ \forall\ \alpha.$ Thus, for any $u\in A_\Phi(G),$ we have, $$\|u_\alpha u\|_{A_\Phi(G)}\leq \|u_\alpha \|_{A_\Phi(G)}\|u\|_{B_\Phi(G)}\leq 2\|u\|_{B_\Phi(G)}.$$ Hence the proof.
\end{proof}
One of the classical results of Reiter states that every closed subgroup of a locally compact abelian group is a set of spectral synthesis for the Fourier algebra $A(G).$ This result is known as the subgroup lemma \cite{Rei}. This result was generalized to locally compact groups by Takesaki and Tatsuuma \cite{TT}. For $1<p<\infty,$ Herz generalized the subgroup lemma to $A_p(G)$ algebras under the assumption that $G$ is amenable. For other generalisations see \cite{DD}. Our next corollary is the subgroup lemma for spectral synthesis. The proof of this is an immediate consequence of Theorem \ref{ABAI} and \cite[Theorem 3.6]{LK}.
\begin{corollary}
Let $G$ be an amenable group and let $\Phi$ satisfy the MA-condition. Then every closed subgroup is a set of spectral synthesis for $A_\Phi(G).$
\end{corollary}

\section{Ideals with bounded approximate identities}
In this section, our aim is to characterize amenable groups in terms of Ditkin sets.

We shall begin this section by introducing some notations. Let $A,~B \subset G$ be closed set of $G.$ Let \begin{eqnarray*}
\mathscr{S}(A,B)&=&\left \{ u\in B_{\Phi}(G) : u(A)=1,u(B)=0~ \right \},\\
s(A,B)&=&\left \{\begin{array}{lc}
\inf\{\|u\|_{B_\Phi(G)}:u\in\mathscr{S}(A,B)\} & \mbox{if }\mathscr{S}(A,B)\neq\emptyset\\
\infty & \mbox{if }\mathscr{S}(A,B)=\emptyset
\end{array}\right.\\
\mathscr{F}(A)&=&\left \{ K \subset G : K~ is ~compact, K\cap A=\emptyset \right \},\\
s_\Phi(A)&=&\sup \left \{ s(A,K): K\in \mathscr{F}(A) \right \}.
\end{eqnarray*}

Our first result is an analogue of \cite[Proposition 3.4]{F3}. This theorem proves the existence of bounded approximate identities with some properties, in certain closed ideals.
\begin{thm}\label{SSAP}
Let $G$ be a amenable locally compact group and let $E$ be a closed subset of $G.$ If $E$ is a set of synthesis for $A_\Phi(G)$ and $s_\Phi(E)<\infty,$ then the ideal $I(E)$ has a bounded approximate identity $\{u_\alpha\}_{\alpha\in\Lambda}$ such that the following holds:
\begin{enumerate}[a)]
\item $\|u_\alpha\|_{A_\Phi(G}\leq 8+4s_\Phi(E)\ \forall\ \alpha\in\Lambda,$
\item $u_\alpha\in A_\Phi(G)\cap C_c(G)\ \forall\ \alpha\in\Lambda,$
\item for every compact subset $K$ of $G$ with $K\cap E=\emptyset,$ there exists a sequence $\{u_n\}$ from $\{u_\alpha\}$ such that for every $u\in A_\Phi(G)$ with $supp(u)\subset K,$ we have $\|uu_n-u\|_{A_\Phi(G)}\leq\frac{1}{n}.$
\end{enumerate}
\end{thm}
\begin{proof}
Since $G$ is amenable, it follows from the proof of Theorem \ref{ABAI}, that $A_\Phi(G)$ possesses an approximate identity $\{u_{K,\epsilon}\}_{\mathcal{F}(E)\times \mathbb{R}^+}$ such that 
\begin{enumerate}[i)]
\item $\|u_{K,\epsilon}\|\leq 2$
\item $supp(u_K,\epsilon)$ is compact and
\item if $v\in A_\Phi(G)$ such that $supp(v)\subseteq K$ then $u_{K,\epsilon}v=\frac{v}{1+\epsilon}.$
\end{enumerate}
Since $s_\Phi(E)$ is finite, there exist $u_K\in \mathscr{S}(E,K)$ such that $\|u_K\|_{B_\Phi(G)}\leq s_\phi(E)+1.$ Let $v_{K,\epsilon}=u_{K,\epsilon}-u_{K,\epsilon}u_K.$ It is clear that $v_{K,\epsilon}\in I(E).$ This $\{v_{K,\epsilon}\}_{(K,\epsilon)\in\mathcal{F}(E)\times\mathbb{R}^+}$ will satisfy the requirements of the theorem.
\end{proof}
\begin{lemma}\label{CSSS}
Let $K$ be a compact subgroup of a locally compact group $G.$ Then $s_\Phi(K)$ is finite.
\end{lemma}
\begin{proof}
Let $C$ be a compact subset of $G$ such that $C\cap K=\emptyset.$ Choose an open neighbourhood $U$ of $e$ such that $U$ is symmetric, relatively compact and $C^{-1}K\cap KU^2=\emptyset.$ Let $u(x)=\frac{1}{|KU|}\chi_{KU}*\check{\chi}_{KU}(x).$ Now, it is clear that $u(e)=1,\|u\|_{B_\Phi(G)}\leq1.$ Further, note that $u$ is 1 on $K$ and 0 on $C,$ i,e., $u\in\mathscr{S}(K,C).$ Hence the proof.
\end{proof}
As an immediate consequence we have the following corollary.
\begin{corollary}\label{SBAI}
Let $G$ be a locally compact amenable group. Then for each $x\in G,$ $I(\{x\})$ contains a bounded approximate identity.
\end{corollary}
\begin{proof}
The proof of this follows from \cite[Theorem 3.6]{LK}, Theorem \ref{SSAP} and Lemma \ref{CSSS}.
\end{proof}
Here is the characterization of amenable groups in terms of the Ditkin sets.
\begin{corollary}\label{SDA}
Let $G$ be a locally compact group and let $\Phi$ satisfy the MA-condition. Then $G$ is amenable if and only if $A_\Phi(G)$ is a strong Ditkin algebra.
\end{corollary}
\begin{proof}
The proof of this follows from Theorem \ref{ABAI} and Corollary \ref{SBAI}.
\end{proof}

\section{Weak factorization and cofinite ideals}
In this section, we characterize amenable groups in terms of weak factorization and cofinite ideals.

For $A_p(G)$ algebras, the following theorem was proved by Losert \cite{L1}.
\begin{thm}\label{AWF}
Let $G$ be a locally compact group $G$ and let $\Phi$ satisfy the MA condition. Then $G$ is amenable if and only if $A_\Phi(G)$ factorizes weakly.
\end{thm}
\begin{proof}
Let $G$ be amenable. Then the if part follows from the Cohen's factorization theorem. We shall now prove the converse. Suppose that $A_\Phi(G)$ weakly factorizes. Note that $A_\Phi(G)$ is a self-adjoint Banach algebra. Thus, by \cite[Theorem 1.3]{FGL}, there exists $c>0$ such that for each compact subset $C$ of $G$ there exists a positive function $u\in A_\Phi(G)$ such that $u\geq 1$ on $C$ and $\|u\|_{A_\Phi(G)}\leq c.$ Observe that, for $\phi\in C_c^+(G),$ the norm of the convolution operator $\|L_\phi\|_{CV_\Psi(G)}$ is equal to the norm of the linear functional $$v\mapsto\int v(x)\phi(x)\ dx$$ on $A_\Phi(G).$ Thus, $\left|\int\phi(x)\ dx\right|\leq c\|L_\phi\|_{CV_\Psi(G)},$ which implies that $\|\phi\|_1^n\leq c\|L_\phi\|^n_{CV_{\Psi}(G)}$ and hence it follows that $\|\phi\|_1\leq \|L_\phi\|_{CV_\Psi(G)}.$ Now proceeding as in the proof of the converse of Theorem \ref{ABAI}, one can show that $G$ is amenable.
\end{proof}
Before we proceed to our next characterization, here are some preparatory lemmas.
\begin{lemma}\label{FSS}
Let $G$ be a amenable group and let $\Phi$ satisfy the MA-condition. Then every finite subset is a set of spectral synthesis.
\end{lemma}
\begin{proof}
The proof of this is an immediate consequence of \cite[Theorem 39.24]{HR} and Corollary \ref{SDA}.
\end{proof}
\begin{lemma}\label{NANF}
Let $G$ be a non-amenable locally compact group and let $I=I(\{e\}).$ Then $I^2$ is not closed in $A_\Phi(G),$ where $\Phi$ satisfies the MA-condition.
\end{lemma}
\begin{proof}
Using \cite[Theorem 3.6]{LK} and Theorem \ref{AWF}, the proof of this follows similar lines as in \cite[Lemma 5.7]{F1}.
\end{proof}
Our next result is the characterization of amenable groups in terms of cofinite ideals.
\begin{thm}\label{ACIH}
Let $G$ be a locally compact group and let $\Phi$ satisfy the MA condition. Then the following are equivalent:
\begin{enumerate}[a)]
\item $G$ is amenable.
\item Every cofinite ideal in $A_\Phi(G)$ is of the form $I(E)$ for some finite subset $E$ of $G.$
\item Each homomorphism from $A_\Phi(G)$ with finite dimensional range is continuous.
\end{enumerate}
\end{thm}
\begin{proof}
a) $\Rightarrow$ b). Let $G$ be amenable and let $I$ be a cofinite ideal in $A_\Phi(G).$ By \cite[Theorem 2.3]{DW}, it is enough to show that every closed cofinite ideal is idempotent. So, let us assume that $I$ is a closed cofinite ideal in $A_\Phi(G).$ Since $I$ is cofinite, the zero set $Z(I)$ is finite and hence, by Lemma \ref{FSS}, $Z(I)$ is a set of spectral synthesis. Thus, it follows that $I=I(Z(I)).$ Further, by Theorem \ref{SSAP}, $I(Z(I))$ has a bounded approximate identity and hence it follows from Cohen's factorization theorem that $I$ is idempotent.

b) $\Rightarrow$ a) follows from \cite[Theorem 2.3]{DW} and Lemma \ref{NANF}. The equivalence of b) and c) follows again from \cite[Theorem 2.3]{DW}.
\end{proof}

\section{Derivations and splittings}
In this section, we characterize amenable groups in terms of continuous derivations. Next we study algebraic splittings and strong splittings of the extensions of the algebra $A_\Phi(G)$ in the spirit of \cite{M}. 

We begin this section by showing the existence of a discontinuous derivation. The proof of this Lemma follows from Lemma \ref{NANF} and \cite[Pg. 402]{DW}.
\begin{lemma}\label{NADD}
Let $G$ be a nonamenable group and let $\Phi$ satisfy the MA-condition. Then there exists a discontinuous derivation of $A_\Phi(G)$ into a finite dimensional commutative Banach $A_\Phi(G)$-bimodule.
\end{lemma}
\begin{lemma}\label{AICD}
Let $G$ be a amenable group and let $I$ be a closed ideal in $A_\Phi(G)$ of infinite codimension. Then there exists sequences $\{u_n\},\{v_n\}$ in $A_\Phi(G)$ such that $u_nv_1\ldots v_{n-1}\notin I$ but $u_nv_1\ldots v_{n-1}u_1v_n\in I$ for all $n\geq 2.$
\end{lemma}
\begin{proof}
As in the proof of Theorem \ref{ACIH}, one can show that the zero set $Z(I),$ of an ideal of infinite codimension, is infinite. Now the remaining proof follows exactly as in the proof of \cite[Lemma 2]{F2}.
\end{proof}
Here is the characterization of amenable groups in terms of continuous derivations.
\begin{thm}
Let $G$ be a locally compact group and let $\Phi$ satisfy the MA-condition. Then the following are equivalent:
\begin{enumerate}[a)]
\item Every derivation of $A_\Phi(G)$ into a Banach $A_\Phi(G)$-bimodule is continuous.
\item Every derivation of $A_\Phi(G)$ into a finite dimensional commutative Banach $A_\Phi(G)$-bimodule is continuous.
\item $G$ is amenable.
\end{enumerate}
\end{thm}
\begin{proof}
a) $\Rightarrow$ b) is trivial and b) $\Rightarrow c)$ follows from Lemma \ref{NADD}. We shall now prove c) $\Rightarrow$ a). In order to prove this, it is enough to verify the conditions of \cite[Theorem 2]{J} for a closed cofinite ideal of $A_\Phi(G).$ But this follows from Theorem \ref{ACIH} and Lemma \ref{AICD}.
\end{proof}

Before we proceed to study algebraic and strong splittings, here are some notations. Let $x,y\in G\cup\{0\}.$ For $z\in\mathbb{C}$ and $u\in A_\Phi(G),$ define $u.z$ and $z.u$ as follows:
\begin{eqnarray*}
u.z:&=&u(x)z\mbox{ for a non-zero } x\mbox{ and } u.z=0\mbox{ otherwise}\\
z.u:&=&u(y)z\mbox{ for a non-zero } y\mbox{ and } z.u=0\mbox{ otherwise}.
\end{eqnarray*}
Note that the above left and right action turns $\mathbb{C}$ into a $A_\Phi(G)$-bimodule. In order to emphasize the role of $x$ and $y,$ we shall denote this bimodule as $\mathbb{C}_{x,y}.$

A linear functional m on $ PM_{\Psi}(G) $ is called a mean if  $ \|m\|=m(I)=1.$ A mean $m$ on $ PM_{\Psi}(G) $ is said to be topologically invariant if $ u .m= u(e)m~~~ \forall ~u \in A_{\Phi}(G)  $, that is,$$ \langle T ,u.m \rangle =  \langle u.T ,m \rangle = u(e)\langle T ,m \rangle ~~ \forall ~~T \in PM_{\Psi}(G),~ \forall~u \in A_{\Phi}(G).$$ It is shown in \cite[Corollary 6.2]{LK} that the set of all topological invariant means on $PM_{\Psi}(G)$ is non-empty. As a result, by following the arguments given for \cite[Lemma 3.1]{M}, we have the following Lemma for $A_\Phi(G).$
\begin{lemma}
If $x\in G,$ then there exists an $A_\Phi(G)$-bimodule homomorphisms $\Theta_x:A_\Phi(G)^\prime\rightarrow\mathbb{C}_{x,x}$ such that $\|\Theta_x\|=\Theta_x(\delta_x)=1.$
\end{lemma}

For algebraic and strong splittings of extensions of Banach algebras, we shall refer to \cite{Da}.

As a consequence of the above Lemma along with \cite[Theorem 3.6]{LK} and \cite[Corollary 3.8]{LK}, we have the following Lemma, whose proof is similar to \cite[Lemma 3.4]{M}.
\begin{lemma}\label{fdsess}
Let $X$ be a finite-dimensional Banach $A_\Phi(G)$-bimodule. Suppose that $X$ is also essential as a left module. Then every singular extension of $A_\Phi(G)$ by $X$ splits strongly.
\end{lemma}
\begin{corollary}\label{fdm}
Let $X$ be a finite-dimensional $A_\Phi(G)$-bimodule. Then $X$ is isomorphic to $\underset{i=1}{\overset{n}{\oplus}}\mathbb{C}_{x_i,y_i},$ for some $n\in\mathbb{N}$ and $x_i,y_i\in G\cup\{0\}.$
\end{corollary}
\begin{thm}
Let $\Phi$ satisfy the MA-condition. If $G$ is amenable, then all finite-dimensional extensions of $A_\Phi(G)$ split strongly.
\end{thm}
\begin{proof}
Let $X$ be a finite-dimensional $A_\Phi(G)$-module. As $G$ is amenable, by Theorem \ref{ABAI}, it follows that $X$ is an essential $A_\Phi(G)$-module. Thus, by Lemma \ref{fdsess}, every singular extension of $A_\Phi(G)$ by $X$ splits strongly. Now the conclusion follows from \cite[Corollary 1.9.8]{Da}.

Another proof of this follows from \cite[Theorem 4.18]{BDL}, Theorem \ref{SSAP} and Theorem \ref{ACIH}.
\end{proof}
\begin{proposition}\label{sa}
Suppose that $A_\Phi(G)$ posseses an approximate identity. Then all the singular finite-dimensional extensions of $A_\Phi(G)$ split algebraically.
\end{proposition}
\begin{proof}
Let $X$ be a finite-dimensional $A_\Phi(G)$-bimodule. By Corollary \ref{fdm}, it follows that $\widetilde{H}^2(A_\Phi(G),X)=\underset{i=1}{\overset{n}{\oplus}}\widetilde{H}^2(A_\Phi(G),\mathbb{C}_{x_i,y_i}),$ where $n$ is the dimension of $X.$ Note that if $x_i$ or $y_i$ are non-zero, then by \cite[Pg. 21-22]{BDL}, it follows that $\widetilde{H}^2(A_\Phi(G),\mathbb{C}_{x_i,y_i})=\{0\}.$ Further, by \cite[Proposition 2.9.34]{Da}, it follows that $\widetilde{H}^2(A_\Phi(G),\mathbb{C}_{0,0})=\{0\}.$ Thus the proof follows from \cite[Corollary 2.8.13]{Da}.
\end{proof}

\section*{Acknowledgement}
The first author would like to thank the University Grants Commission, India, for research grant.

\end{document}